\pdfoutput=1
\RequirePackage{ifpdf}
\ifpdf 
\documentclass[pdftex]{sigma}
\else
\documentclass{sigma}
\fi

\newtheorem{Theorem}{Theorem}[section]
\newtheorem{Corollary}[Theorem]{Corollary}

\newtheorem{Proposition}[Theorem]{Proposition}
{\theoremstyle{definition}
\newtheorem{Definition}[Theorem]{Definition}
\newtheorem{Notation}[Theorem]{Notation}

\newtheorem{Remark}[Theorem]{Remark}
}

\def\di{\operatorname{diag}}

\def\vr{\varrho}
\def\a{\alpha}

\def\om{\omega}

\def\vp{\varphi}
\def\ve{\varepsilon}

\def\wh{\widehat}
\def\wt{\widetilde}
\def\ov{\overline}
\def\p{\partial}
\def\prt{\partial}

\def\BC{{\mathbb C}}
\def\BR{{\mathbb R}}

\def\clp{{\mathcal P}}
\def\clr{{\mathcal R}}

\def\clm{{\mathcal M}}

\def\cln{{\mathcal N}}
\def\clu{{\mathcal U}}
\newcommand{\E}{\mathrm{e}}
\newcommand{\I}{\mathrm{i}}
\def\mf{\mathfrak}

\newcommand{\cD}{{\mathcal D}}

\newcommand{\cF}{{\mathcal F}}

\newcommand{\cM}{{\mathcal M}}
\newcommand{\cN}{{\mathcal N}}

\newcommand{\cP}{{\mathcal P}}
\newcommand{\cQ}{{\mathcal Q}}
\newcommand{\cR}{{\mathcal R}}

\newcommand{\cV}{{\mathcal V}}

\numberwithin{equation}{section}

\begin{document}

\allowdisplaybreaks

\newcommand{\arXivNumber}{1405.3500}

\renewcommand{\PaperNumber}{001}

\FirstPageHeading

\ShortArticleName{Initial Value Problems for Integrable Systems on a Semi-Strip}

\ArticleName{Initial Value Problems for Integrable Systems\\ on a Semi-Strip}

\Author{Alexander L.~{SAKHNOVICH}}

\AuthorNameForHeading{A.L.~Sakhnovich}
\Address{Vienna University of Technology, Institute of Analysis and Scientific Computing,\\
Wiedner Hauptstr. 8, A-1040 Vienna, Austria}
\Email{\href{mailto:oleksandr.sakhnovych@tuwien.ac.at}{oleksandr.sakhnovych@tuwien.ac.at}}

\ArticleDates{Received September 01, 2015, in f\/inal form December 28, 2015; Published online January 03, 2016}

\Abstract{Two important cases, where boundary conditions and solutions of the well-known integrable equations
on a semi-strip are uniquely determined by the initial conditions, are rigorously studied in detail. First, the case of
rectangular matrix solutions of the defocusing nonlinear Schr\"odinger equation
with quasi-analytic boundary conditions is dealt with. (The result is new even for a scalar nonlinear Schr\"odinger equation.)
Next, a special case of the nonlinear
optics ($N$-wave) equation is considered.}


\Classification{35Q55; 35Q60; 34B20; 35A02}

\section{Introduction} \label{Intro}

Cauchy problems for wave equations were successfully investigated
using the classical inverse scattering transform, Laplace transforms and also
some other methods.
The theory of initial-boundary value problems (and problems in a quarter-plane
or semi-strip) is somewhat more complicated even for the case of linear wave equations.
Some results, discussions and references on this topic are given
 in \cite{Ash, DeMaS, FoP, GMZ, Krei, SaAF}. The mentioned above results
and discussions are also related to the integrable nonlinear equations.
Although it is impossible to refer here to the whole variety of important publications on the
initial-boundary value problems for integrable wave equations, we would like
to list just some: \cite{BiTa, BW83, CB, DeMaS, Fok, FoP, Holm, KamS, Ka, KaNe, KaSt, Kri, Sab1, Sab3, SaA7, SaSaR, SaL91, Skl}.
Usually (excluding, e.g., sine-Gordon case \cite{KaNe, Kri, SaAg, SaSaR}),
initial-boundary value problems are overdetermined (see \cite{BoFo, Fok, SaSaR}) for such cases of integrable nonlinear equations, where
exact procedures to recover solutions from initial and boundary values exist.
Thus, the reduction of the initial-boundary conditions, which are necessary to recover solutions, becomes crucial for
solving initial-boundary value problems.

In this paper we extend and develop further the work (which was started in \cite{Sa14b}) on the reduction
of the necessary initial-boundary conditions. Namely, a case, where boundary conditions provide direct information about an initial condition,
was investigated in \cite{Sa14b}.
Here we study the cases, where an initial condition provides direct information about boundary conditions,
and the solutions of the well-known nonlinear integrable equations in a semi-strip are uniquely def\/ined by the
initial conditions. (One could speak, perhaps, about additional symmetries of the solutions.)

We consider such situations using the inverse spectral transform approach \cite{Ber, BerG, Kv}.
More precisely, we follow the scheme introduced in \cite{Sa88, SaLev2}, see also \cite[Chapter~12]{Sa99}
and references therein. That is, we describe the evolution of Weyl--Titchmarsh (Weyl) function in terms of the linear-fractional transformations.
The scheme is applicable to various integrable equations
and several interesting uniqueness and existence theorems were proved in this way (see \cite[Chapter~6]{SaSaR}
and references therein for more details). Most of the mentioned above uniqueness and existence theorems
were obtained for the equations with scalar solutions. Here we consider
the matrix defocusing nonlinear Schr\"odinger (defocusing NLS or dNLS) equation
\begin{gather} \label{1.1}
2v_t=\I(v_{xx}-2vv^*v), \qquad v_t:=\frac{\p}{\p t}v ,
\end{gather}
which is equivalent \cite{ZS1, ZS2} to the compatibility condition
\begin{gather} \label{zc}
G_t-F_x+[G,F]=0 \qquad [G,F]:=GF-FG
\end{gather}
of the auxiliary linear systems
\begin{gather} \label{au}
y_x=Gy, \qquad y_t=Fy,
\end{gather}
where
\begin{gather} \label{1.2}
G=\I (zj+jV), \qquad
F=-\I\big(z^2 j+ z jV-\big(\I V_x-jV^2\big)/2\big), \\
\label{1.3}
j = \left[
\begin{matrix}
I_{m_1} & 0 \\ 0 & -I_{m_2}
\end{matrix}
\right], \qquad V= \left[\begin{matrix}
0&v\\v^{*}&0\end{matrix}\right],
\end{gather}
$I_{m_1}$ is the $m_1\times m_1$ identity matrix and $v$ is an $m_1\times m_2$ matrix function.
We will consider dNLS equation on the semi-strip
\begin{gather}  \label{1.4}
{\mathcal D}=\{(x,\, t)\colon 0 \leq x <\infty, \, 0\leq t<a\},
\end{gather}
and we note that the auxiliary system
\begin{gather} \label{1.5}
y_x=Gy=\I(zj+jV)y
\end{gather}
is (for each f\/ixed~$t$) a well-known self-adjoint Dirac system, also called AKNS or Zakharov--Shabat system.
Without changes in notation, we speak about usual derivatives inside domains
and about left or right
(which should be clear from the context) derivatives
on the boundaries, and boundaries of~$\cD$ in particular.
We should mention that, in the usual PDE setting, solutions are often considered in open domains
but, in view of certain regularity of the solutions treated in our paper, the boundaries are included
in~${\mathcal D}$ in our case.

Since $v$ in dNLS \eqref{1.1} is an $m_1 \times m_2$ matrix function,
interesting matrix, vector and multicomponent dNLS equations from~\cite[Chapter~4]{Abl0}
are included in the considered class.

Another equation that we study in this paper is the nonlinear optics (or $N$-wave) equation:
\begin{gather} \label{NW1}
\left[D, \frac{\prt \vr}{\prt t} \right] - \left[\wh D, \frac{\prt \vr}{\prt x} \right]
=
[D, \vr] [\wh D, \vr] - [\wh D, \vr] [ D, \vr], \qquad \vr=\vr^*, \\
\label{NW2} D=
\di \{ d_1,  d_2, \ldots,
 d_m \}, \qquad d_1>d_2>\cdots>d_m>0, \qquad [D,\vr]:=D\vr-\vr D,
\end{gather}
where $\vr(x,t)$ is an $m\times m$ matrix function, $\di \{ d_1, d_2, \ldots \}$ stands for a diagonal matrix with the
entries $d_1, d_2, \dots$ on the main diagonal, and $\wh D>0$ is another diagonal matrix.

First, we obtain the evolution of the Weyl function for the equation~\eqref{NW1}.
 Next, we consider an interesting special case, where (similar to the inequalities for the entries of~$D$)
we have
\begin{gather} \label{NW3}
\wh D=\di \big\{ \wh d_1, \wh d_2, \ldots,
\wh d_m \big\}, \qquad \wh d_1> \wh d_2>\cdots> \wh d_m>0.
\end{gather}

In Section \ref{Prel} we formulate some necessary results on Weyl functions (and their evolution for the dNLS case), in Section~\ref{Quasi}
we recover the boundary conditions for dNLS from an initial condition on a semi-axis
(and in this way we solve the initial value problem for dNLS in a~semi-strip), and Section~\ref{NW}
is dedicated to the evolution of the Weyl function and initial value problem in a semi-strip for the $N$-wave equation.

We note that the theory of Weyl functions (Weyl--Titchmarsh theory) is actively developing in recent years (see, e.g., \cite{ClGe, DaKiS, FKS, GeZi,GKS6, KoSaTe,SaSaR, Sa99, Si0, Te}
and references therein) and its applications
to initial-boundary value problems are of growing interest.

As usual, $\BR$ stands for the real axis,
$\BR_+=(0, \infty)$,
$\BC$ stands for the complex plane, and
$\BC_+$ for the open semi-plane $\{z\colon \Im(z)>0\}$.
We say that~$v(x)$ is locally summable if its entries are summable on all f\/inite intervals of $[0, \infty)$.
We say that~$v$ is continuously dif\/ferentiable
if~$v$ is dif\/ferentiable and its f\/irst derivatives are continuous.
The notation~$\| \cdot \|$ stands for the~$l^2$ vector norm or the induced matrix norm.
The partial derivative~$f_{xt}$ stands for $\p f_x / \p t$ and, correspondingly, $f_{tx} = \p f_t / \p x$.

\section{Preliminaries} \label{Prel}

\subsection{Dirac system and dNLS}

We denote by $u$ the fundamental solution of system~\eqref{1.5} normalized by
the condition
\begin{gather*}  
u(0,z)=I_m, \qquad m=m_1+m_2.
\end{gather*}
\begin{Definition} \label{CyW2} Let Dirac system~\eqref{1.5} on $[0, \infty)$
 be given and assume that $V$ is locally summable.
Then the Weyl function $\vp$ is an $m_2 \times m_1$ holomorphic matrix function, which satisf\/ies the inequality
\begin{gather}  \label{2.0}
\int_0^{\infty}
\begin{bmatrix}
I_{m_1} & \vp(z)^*
\end{bmatrix}
u(x,z)^*u(x,z)
\begin{bmatrix}
I_{m_1} \\ \vp(z)
\end{bmatrix}dx< \infty .
\end{gather}
\end{Definition}

The following proposition is proved in \cite{FKRSp1} (see also \cite[Section~2.2]{SaSaR}).
\begin{Proposition} The Weyl function always exists and it is unique.
\end{Proposition}

In order to construct the Weyl function, we introduce a class of $m \times m_1$ matrix functions
$\clp(z)$, which are an immediate analog of the classical pairs
of parameter matrix functions. Namely, the matrix functions
$\clp(z)$ are meromorphic in $\BC_+$ and satisfy
(excluding, possibly, a discrete set of points)
the following relations
\begin{gather}\label{2.1}
\clp(z)^*\clp(z) >0, \qquad \clp(z)^* j \clp(z) \geq 0, \qquad z\in \BC_+.
\end{gather}
It is said that $\clp(z)$ are nonsingular (i.e., the f\/irst inequality in~\eqref{2.1} holds)
and with property-$j$ (i.e., the second inequality in~\eqref{2.1} is valid).
Relations~\eqref{2.1} imply (see, e.g., \cite{FKRSp1}) that
\begin{gather*}  
\det \big(\begin{bmatrix}
I_{m_1} & 0
\end{bmatrix}u(x,z)^{-1}\clp(z)\big)\not= 0.
\end{gather*}
\begin{Definition} \label{set}
The set $\cln(x,z)$ of M\"obius transformations is the set of values $($at the f\/ixed $x \in [0, \infty)$, $z\in \BC_+)$
of matrix functions
\begin{gather*}
\vp(x,z,\clp)=\begin{bmatrix}
0 &I_{m_2}
\end{bmatrix}u(x,z)^{-1}\clp(z)\big(\begin{bmatrix}
I_{m_1} & 0
\end{bmatrix}u(x,z)^{-1}\clp(z)\big)^{-1},
\end{gather*}
where $\clp(z)$ are nonsingular matrix functions
 with property-$j$.
 \end{Definition}

 \begin{Remark}\label{MB}
 It was shown in \cite{FKRSp1} that a family $\cln(x,z)$, where $x$ increases to inf\/inity and $z$ is f\/ixed
 $(z \in \BC_+)$,
is a family of embedded matrix balls such that the right semi-radii
 are uniformly bounded and the left semi-radii tend to zero. $($Recall that the $m_2\times m_1$ matrix ball, or Weyl matrix ball, with the center~$\clm$, the left semi-radius~$\clr_l$
 and the right semi-radius $\clr_r$ is the set of
 $m_2\times m_1$ matrices~$\om$
 which may may be presented
 in the form $\om = \clm + \clr_l \clu \clr_r$, where~$\clu$ are contractive $m_2\times m_1$ matrices.$)$
 \end{Remark}

 \begin{Proposition}[\cite{FKRSp1}] \label{PnW1} Let Dirac system \eqref{1.5} on $[0, \infty)$
 be given and assume that $V$ is locally summable.
 Then, the sets $\cln(x,z)$
 are well-defined. There is a unique matrix function
 $\vp(z)$ defined in $\BC_+$ and such that
\begin{gather}  \label{2.3}
\{\vp(z) \}=\bigcap_{x<\infty}\cln(x,z).
\end{gather}
This function is analytic and non-expansive $($i.e., contractive$)$.
Furthermore, this function coincides with the Weyl function of system~\eqref{1.5}.
 \end{Proposition}

Formula \eqref{2.3} is supplemented by the asymptotic relation
\begin{gather}  \label{2.3!}
\vp(z)=\lim_{b \to \infty}\vp_b(z),
\end{gather}
which is valid for any set of functions $\vp_b(z)\in \cln(b,z)$. Relation~\eqref{2.3!} follows from~\eqref{2.3}
and Remark~\ref{MB} (see also \cite[Remark~2.24]{SaSaR}).

Next, we consider the famous compatibility condition (zero curvature equation)~\eqref{zc}.
\begin{Proposition}[\cite{SaAF}] \label{TmM}
Let some $m \times m$ matrix functions $G$ and $F$ and their derivatives~$G_t$ and~$F_x$ exist on the semi-strip~$\cD$,
let $G$, $G_t$ and~$F$ be continuous with respect to~$x$ and~$t$ on~$\mathcal{D}$,
 and let~\eqref{zc}
hold. Then, we have the equality
\begin{gather} \label{2.4}
u(x,t,z)R(t,z)=R(x,t,z)u(x,0,z), \qquad R(t,z):=R(0,t,z),
\end{gather}
where $u(x,t,z)$ and $R(x,t,z)$ are normalized fundamental solutions given, respectively, by
\begin{gather} \label{2.5}
u_x=Gu, \qquad u(0,t,z)=I_m; \qquad R_t=FR, \qquad R(x,0,z)=I_m.
\end{gather}
The equality \eqref{2.4} means that the matrix function
\begin{gather*}
y(x,t,z)=u(x,t,z)R(t,z)=R(x,t,z)u(x,0,z)
\end{gather*}
satisf\/ies on $\cD$ both systems \eqref{au}. Moreover, the fundamental solution~$u$ admits the factorization
\begin{gather} \label{2.6}
u(x,t,z)=R(x,t,z)u(x,0,z)R(t,z)^{-1}.
\end{gather}
\end{Proposition}

Proposition \ref{TmM} and formula \eqref{2.3!} yield~\cite{Sa14b} the following evolution theorem.
 \begin{Theorem}\label{evol} Let an $m_1 \times m_2$ matrix function $v(x,t)$
 be continuously differentiable on $\cD$
 and let $v_{xx}$ exist.
 Assume that $v$ satisfies the dNLS equation \eqref{1.1} as well as the following inequalities
 $($for all $0 \leq t<a$ and some values $M(t)\in \BR_+)$:
 \begin{gather} \label{2.7}
 \sup_{x \in \BR_+, \, 0\leq s \leq t}\|v(x,s)\| \leq M(t).
\end{gather}
 Then, the evolution $\vp(t,z)$ of the Weyl functions of Dirac systems \eqref{1.5} is given $($for $\Im(z)>0)$ by the equality
 \begin{gather*} 
\vp(t,z)=\big(R_{21}(t,z)+R_{22}(t,z)\vp(0,z)\big)
\big(R_{11}(t,z)+R_{12}(t,z)\vp(0,z)\big)^{-1}.
\end{gather*}
\end{Theorem}

\begin{Remark}\label{RkInvPr} According to \cite{Sa14}, the Dirac system $y_x=Gy$ $($where $G$ is given by~\eqref{1.2} and~\eqref{1.3}
and~$v$ is locally square summable$)$
is uniquely recovered from the Weyl function $\vp$. In other words, $v$ is uniquely recovered from~$\vp$, see the procedure in \cite[Theorem~4.4]{Sa14}.
The case of a~more smooth $($i.e., locally bounded$)$ $v$ was dealt with in~\cite{SaSaR}, see also references therein.
\end{Remark}

\subsection{Auxiliary linear systems for nonlinear optics equation}\label{subs2.2}

The nonlinear optics ($N$-wave) equation \eqref{NW1} is the compatibility condition
of the systems \eqref{au}, where
\begin{gather} \label{NW4}
G(x,t,z)=\I z D-\zeta(x,t), \qquad F(x,t,z)=\I z \wh D-\wh \zeta(x,t); \nonumber\\ \zeta=[D, \vr], \qquad \wh \zeta =[\wh D, \vr];
\end{gather}
see \cite{ZaMa} for the case $N=3$ and \cite{AbH} for $N>3$. We shall need
some preliminary results on the Weyl theory of the auxiliary system $y_x=Gy$
from \cite[Chapter~4]{SaSaR} (see also~\cite{SaA17}). The normalized fundamental solution~$w$
of such a system is def\/ined by the formula
\begin{gather} \label{NW5}
w_x(x,z)=\big(\I z D-\zeta(x)\big)w(x,z), \qquad w(0,z)=I_m, \qquad \zeta =- \zeta^* .
\end{gather}
Here and later we assume that $D$ is a f\/ixed matrix satisfying \eqref{NW2}. We consider
system~\eqref{NW5} with locally bounded potentials~$\zeta$, that is, potentials satisfying (for each $l<\infty$)
the inequality
\begin{gather}
  \sup\limits_{0 < x < l}
   \bigl\|
    \zeta (x)
   \bigr\|
  < \infty.
 \label{NW6}
\end{gather}
\begin{Definition} \label{NWDn2}
 A generalized Weyl function $($GW-function$)$ of system
\eqref{NW5}, where $\zeta$ is locally bounded, is an $m \times m$ matrix function $\vp$ such that for some $M>0$ it is analytic in the domain
$\BC^{-}_M = \{ z \colon \Im(z) < - M \}$ and the
inequality
\begin{gather}
  \sup\limits_{x \leq l, \, \Im (z) < -M}
   \bigl\|
    w(x,z)
    \varphi (z)\exp \{-\I z x D\}
   \bigr\|
 < \infty
 \label{NW7}
\end{gather}
holds for each
$l < \infty$.
\end{Definition}

\begin{Remark} We note that the Weyl function of the system \eqref{NW5} is def\/ined $\big($for $\zeta$ bounded on $[0,\, \infty)\big)$
by the analog \eqref{NWd6} of the inequality \eqref{2.0} and by normalization conditions \eqref{NWd7}. If~$\zeta$ is bounded on $[0,\, \infty)$, this Weyl function
coincides with the normalized {\rm GW}-function.
$($See the discussion after formula~\eqref{NWd7} and Def\/inition~\ref{NWDn2} of the {\rm GW}-function above.$)$ The fact that the Weyl function satisf\/ies~\eqref{NW7} explains the term ``generalized Weyl function'' $($or ``{\rm GW}-function''$)$.
\end{Remark}

The inverse spectral problem (ISpP) for system
\eqref{NW5} is the problem of recovering (from a~GW-function $\varphi$)
a potential
$\zeta (x) = - \zeta^* (x)$
 such that~\eqref{NW7} is valid and the diagonal entries~$\zeta_{kk}$ of~$\zeta$ equal zero $($i.e., $\zeta_{kk} \equiv 0)$.

\begin{Notation}
The notation $\mf M$
stands for
an operator mapping the pair $D$ and~$\varphi $
into the corresponding potential~$\zeta$
$($i.e., ${\mf M} (D, \varphi) = \zeta)$. In other words, ${\mf M} (D, \varphi)$ stands for a solution of the ISpP.
\end{Notation}

\noindent The following theorem (i.e., \cite[Theorem 4.8]{SaSaR}) is valid.
\begin{Theorem}\label{NWTmUniq}
For any
matrix function
$\varphi (z)$ which is
analytic and bounded in
$\BC^{-}_M$
and has the
property
\begin{gather}
  \int_{- \infty}^\infty
   \bigl(
    \varphi (z) - I_m
   \bigr)^*
   \bigl(
    \varphi (z) - I_m
   \bigr)
  d \xi
  < \infty,
\qquad
  z = \xi + \I \eta,
  \qquad
  \eta <- M ,
 \label{NW8}
\end{gather}
there is at most one solution of the ISpP.
\end{Theorem}
Our next corollary for systems on $(0,l)$ is immediate from the proof (see \cite[pp.~108--109]{SaSaR}) of Theorem~\ref{NWTmUniq}.

\begin{Corollary} \label{CyUniq} Assume that \eqref{NW6} is valid and let relations \eqref{NW7} and \eqref{NW8} hold
for an analytic and bounded matrix function~$\vp$. Then, $\zeta(x)$ is uniquely defined on~$(0,l)$.
\end{Corollary}

\begin{Remark}\label{NWInvPr} Under somewhat stronger $($than in~\eqref{NW8}$)$
restrictions, the solution of the ISpP always exists. Namely, if for some matrix $\phi_0$ and some $M>0$
 we have
\begin{gather} \label{NW9}
\sup_{\Im(z)<-M}
   \big\| \,
   z ( \varphi (z) - I_m) \,
   \big\|
   < \infty, \quad \det \, \varphi (z) \not= 0 \quad {\mathrm{for}} \quad \Im(z) <-M,
     \\
   (\xi+\I\eta)\big( \varphi (\xi+\I\eta) - I_m - \phi_0 / (\xi+\I \eta) \big)
   \in L_{m \times m}^2 (- \infty, \infty)
   \quad \text{for all f\/ixed} \
   \eta<-M,
 \label{NW10}
\end{gather}
then ${\mf M} (D, \varphi)$ is constructed in \cite[Theorem~4.10]{SaSaR}.
\end{Remark}

The situation becomes simpler when
$\zeta(x)$ is uniformly bounded on $[0, \infty)$, that is,
\begin{gather}
  \sup\limits_{0 < x < \infty}
   \bigl\| \zeta (x) \bigr\|
 \leq M_0 .
 \label{NWd8}
\end{gather}
We recall that  a Weyl function of system \eqref{NW5} is introduced in another way
than a GW-function. Namely, a Weyl function
 is an analytic $m \times m$ matrix function~$\varphi (z)$,
satisfying for certain
$M > 0$
and
$r > 0$
and for all~$z$
from the domain
$\BC^{-}_M = \{ z \colon \Im(z) < - M \}$
the inequality
\begin{gather}
  \int_0^\infty
      \exp \{\I \overline z x D \}
     \varphi (z)^*
     w (x,z)^*
     w (x,z)
     \varphi (z)
     \exp
      \bigl\{{-}
        (\I z D + r I_m )x
      \bigr\}
  dx
 < \infty,
 \label{NWd6}
\end{gather}
and the normalization conditions on the entries $\varphi_{ij} (z)$:
\begin{gather}
   \varphi_{ij} (z) \equiv 1 \qquad \mbox{for} \quad i= j, \qquad
   \varphi_{ij} (z) \equiv 0 \qquad \mbox{for} \quad i > j.
 \label{NWd7}
\end{gather}
When \eqref{NWd8} holds,
a Weyl function of system \eqref{NW5} exists and is unique.
Moreover, for that case~$\vp$ is the unique GW-function (of system \eqref{NW5} with the given $\zeta$)
satisfying normalization conditions~\eqref{NWd7}. In order to construct this Weyl (and simultaneously GW-) function
we use matrices $j$ of the form~\eqref{1.3} for each $1\leq m_1 < m$, that is, we set
\begin{gather*}
J_k:= \left[
     \begin{matrix}
       I_k & 0   \\
       0 & -I_{m-k}
     \end{matrix}
    \right], \qquad   1\leq k < m.
\end{gather*}
Now, the Weyl function is constructed \cite[pp.~103--106]{SaSaR}
in the following way.

First, for each $ k $, we introduce a class of
$m \times (m-k)$ matrix functions $\cQ_k$, which are meromorphic in some semi-plane $\BC^{-}_{M_k}$ and satisfy the inequalities
\begin{gather}
\cQ_k (z)^*
  \cQ_k(z)
 > 0, \qquad
  \cQ_k (z)^*
  J_k
  \cQ_k (z)
\leq 0,
 \label{NW12}
\end{gather}
excluding, possibly, isolated points. These $\cQ_k$ are called nonsingular with property-$J_k$.
Assu\-ming $M_k>2M_0/(d_k-d_{k+1})$ and using \eqref{NW12}, one can show that the matrix function
\begin{gather}
  \psi_k (x,z)
= \begin{bmatrix}I_k &0 \end{bmatrix}
 w (x,z)^{-1}
 \cQ_k (z)
 \left(
  \begin{bmatrix}0 &I_{m-k} \end{bmatrix}
  w (x,z)^{-1}
  \cQ_k(z)
 \right)^{-1}
 \label{NW13}
\end{gather}
is well-def\/ined for $x \geq 0$, $z \in \BC^{-}_{M_k}$, and satisf\/ies the inequality
\begin{gather}
\begin{bmatrix} \psi_k (z)^* &I_{m-k} \end{bmatrix}
 J_k
 \left[
  \begin{matrix}
    \psi_k (z) \\
    I_{m-k}
  \end{matrix}
 \right]
\leq 0 , \qquad {\mathrm{i.e.,}} \qquad \psi_k (z)^*
  \psi_k (z)
\leq I_{m-k} .
 \label{NW14}
\end{gather}
The set of matrices $\psi_k(x,z)$ given by \eqref{NW13}, where $x$ and $z$ are f\/ixed and matrix functions $\cQ_k(z)$ are nonsingular with property-$J_k$, is denoted
by $\cN_k(x,z)$. These sets are embedded and have a~point limit, that is, similar to \eqref{2.3} and \eqref{2.3!} we have
 \begin{gather}
  \breve \psi_k(z) = \bigcap\limits_{x < \infty}
           {\cal N}_k (x,z), \qquad \breve \psi_k(z)=\lim_{x\to \infty} \psi_k(x,z),
\qquad
  z \in \BC_{M_k}^{-} .
 \label{NW15}
\end{gather}
In this way we recover the $(k+1)$th column of the Weyl function $\varphi$. More precisely, we have
\begin{gather}
  \left\{
   \varphi_{i, k+1}(z)
  \right\}_{i=1}^k
 = \breve \psi_k (z)\begin{bmatrix} 1\\ 0 \\ \ldots \\0
 \end{bmatrix}, \qquad \Im(z)<-M \label{NW16}
\end{gather}
for any $M>\max\limits_{1\leq k<m}\big(2M_0/(d_k-d_{k+1})\big)$. There is also an inverse transformation \cite[Remark~4.6]{SaSaR}, which expresses $\breve \psi_k$ via $\varphi$:
\begin{gather}
  \breve\psi_k^{\,} (z)
= \begin{bmatrix} I_k & 0\end{bmatrix}
 \varphi (z)
 \left[
  \begin{matrix}
    0  \\
    I_{m-k}
  \end{matrix}
 \right]
 \biggl(
 \begin{bmatrix} 0 & I_{m-k}\end{bmatrix}
  \varphi (z)
  \left[
  \begin{matrix}
    0  \\
    I_{m-k}
  \end{matrix}
  \right]
 \biggr)^{-1} .
 \label{NW17}
\end{gather}
Finally, we will need a representation of $w(x,z)$ on intervals $0\leq x\leq l$, $l<\infty$, \cite[equation~(4.37)]{SaSaR} :
\begin{gather}
  w(x,z) = \exp \{ \I z x D\}
    + \int_{d_m x}^{d_1 x}
           \exp \{\I zt\}
     N(x,t)
     dt, \qquad \sup_{x\leq l} \| N(x,t) \| < \infty.
 \label{NW18}
\end{gather}

\section{NLS with quasi-analytic boundary conditions}\label{Quasi}

First publications on initial-boundary value problems for integrable systems (see, e.g., \cite{Kv, KaNe}) appeared only several years after
the great breakthrough for Cauchy problems for such systems.
Interesting numerical \cite{ChuX}, uniqueness \cite{BW83, Ton} and local existence \cite{Holm, Kri} results followed.
Special {\it linearizable} cases of boundary conditions were found using symmetrical reduction \cite{Skl} or BT (B\"acklund transformation) method
\cite{BiTa, Hab}. Global existence results for Dirichlet and Neumann initial-boundary value problems (for cubic NLS equations)
were obtained using PDE methods in~\cite{CB} and~\cite{Kai}, respectively. Interesting approaches were developed by
D.J.~Kaup and H.~Steudel~\cite{KaSt}, by P.~Sabatier (elbow scattering)~\cite{Sab1, Sab3} and by
 A.S.~Fokas (global relation method)
 \cite{Fok}, see also some discussions on the corresponding dif\/f\/iculties and open problems in~\cite{Ash, BoFo}.
Since many publications were dedicated to the initial-boundary value problems for NLS equations, it is of special interest
that a wide class of solutions of NLS in a semi-strip is uniquely determined by the initial condition.

We note that the case of quasi-analytic boundary (or initial) conditions is important also because related suggestions
that initial and boundary conditions (or even solutions) belong to the so-called Schwartz class of functions are often used
for simplicity (see, e.g., \cite{Fok0}). Our result shows that one should be rather careful with such suggestions,
so that they agree with the established interrelations between initial and boundary conditions.

Recall that the domain $\cD$ is def\/ined in~\eqref{1.4}.

\begin{Notation}\label{NnB1}
We consider $m_1 \times m_2$ matrix functions $v(x,t)$,
which are continuously dif\/fe\-ren\-tiab\-le and are
such that~$v_{xx}$ exists
on the semi-strip~$\cD$.
Moreover, we require that for each $k$ there is a~value $\ve_k=\ve_k(v)>0$ such that $v$
is~$k$ times continuously dif\/ferentiable with respect to $x$ in the square
\begin{gather} \label{B2'}
\cD({\ve_k})=\{(x, t)\colon  0\leq x\leq \ve_k, \, 0\leq t\leq \ve_k\}, \qquad \cD({\ve_k})\subset \cD.
\end{gather}
The class of such functions $v(x,t)$ is denoted by $C_{\ve}(\cD)$.
\end{Notation}

\noindent Without loss of generality, we assume that the values $\ve_k$ in \eqref{B2'} monotonically decrease.
\begin{Proposition} \label{prequa} Assume that $v\in C_{\ve}(\cD)$ satisfies the dNLS equation \eqref{1.1} on $\cD$.
Then, for each integer $r\geq 0$ and values $0\leq k\leq r$, the functions $\left(\frac{\p^k}{\p t^k}v\right)(x,0)$ and $\left(\frac{\p^k}{\p t^k}v_x\right)(x,0)$
may be uniquely recovered $($on the interval $0\leq x\leq \ve_
{4r})$ from the initial condition
\begin{gather} \label{IC1}
v(x,0)=\cV(x).
\end{gather}
Moreover, on the domain $\cD(\ve_{4(r+1)})$ the functions $\left(\frac{\p^{k+1}}{\p t^{k+1}}v\right)$, $\left(\frac{\p^{k+1}}{\p t^{k+1}}v\right)_x$
and $\left(\frac{\p^{k+1}}{\p t^{k+1}}v\right)_{xx}$
exist and are continuous, and the equalities
\begin{gather} \label{B5}
\left(\frac{\p^k}{\p t^k}v\right)_{xt}= \left(\frac{\p^k}{\p t^k}v\right)_{tx}, \qquad
\left(\frac{\p^k}{\p t^k}v\right)_{xxt}= \left(\frac{\p^k}{\p t^k}v\right)_{txx}
\end{gather}
hold, whereas both sides of these equalities are again continuous.
 For $0\leq k\leq r$ and $0\leq \ell\leq s$ the functions $\left(\frac{\p^{\ell}}{\p x^{\ell}}\frac{\p^{k+1}}{\p t^{k+1}}v\right)$
exist and are continuous in the domains $\cD(\ve_{s+4(r+1)})$.
\end{Proposition}

In order to prove this proposition, we need a stronger version of the well-known Clairaut's (or Schwarz's) theorem on mixed derivatives.
We need this version for the closed square $\cD(\ve)$ (as in Proposition~3.2 from \cite{Sa14b}), which
statement easily follows from the proofs of the mixed derivatives theorem for open domains (see, e.g.,~\cite{See}).
\begin{Proposition}[\cite{Sa14b}]\label{PnB2} If the functions $f$, $f_t$ and $f_{tx}$ exist and are continuous on $\cD(\ve)$
and the derivative $f_x(x, 0)$ exists for $0\leq x \leq \ve$, then $f_x$ and $f_{xt}$ exist on $\cD(\ve)$
and $f_{xt}=f_{tx}$.
\end{Proposition}
\begin{proof}[Proof of Proposition \ref{prequa}]
We prove Proposition \ref{prequa} by induction. First, consider the case $r=0$.
Clearly, $v(x,0)$ and $v_x(x,0)$ are given by the initial condition \eqref{IC1}.
Since the right-hand side of \eqref{1.1} is two times continuously dif\/ferentiable with respect to $x$ in $\cD(\ve_4)$,
we derive that $v_t$, $v_{tx}$ and $v_{txx}$ exist and are continuous in $\cD(\ve_4)$:
\begin{gather*}
2v_t=\I(v_{xx}-2vv^*v), \qquad 2v_{tx}=\I\left(v_{xxx}-2\frac{\p}{\p x}(vv^*v)\right), \\ 2v_{txx}=\I\left(v_{xxxx}-2\frac{\p^2}{\p x^2}(vv^*v)\right).
\end{gather*}
Moreover, putting $f=v$ we see that conditions of Proposition~\ref{PnB2}
are fulf\/illed and the f\/irst equality in~\eqref{B5} holds for $k=0$. Putting $f=v_x$ and taking into account that the f\/irst equality
in \eqref{B5} yields $v_{xt}=v_{tx}$ and $v_{xtx}=v_{txx}$, we see that conditions of Proposition~\ref{PnB2} hold also for $f=v_x$.
That is, $v_{xxt}$ exists and equals $v_{xtx}=v_{txx}$. Thus, \eqref{B5} is proved for $k=0$ (i.e., for $r=0$). In view of~\eqref{1.1}, it is immediate that the
last statement of Proposition~\ref{prequa} is also valid for $r=0$.

Next, assuming that the statements of Proposition \ref{prequa} hold for all $0\leq r \leq r_0$, let us prove them for $r=r_0+1$.
Dif\/ferentiating both sides of \eqref{1.1} $r_0$ times with respect to $t$ and taking into account (for $r_0>0$ and $k\leq r_0-1$) the second
equality in \eqref{B5}, we express $\left(\frac{\p^{r_0+1}}{\p t^{r_0+1}}v\right)(x,0)$ via derivatives which we already know.
Then, from the f\/irst equality in~\eqref{B5}, we obtain the formula $\left(\frac{\p^{r_0+1}}{\p t^{r_0+1}}v_x\right)(x,0)=\left(\frac{\p^{r_0+1}}{\p t^{r_0+1}}v\right)_x(x,0)$
and an expression for $\left(\frac{\p^{r_0+1}}{\p t^{r_0+1}}v_x\right)(x,0)$ follows.

 Dif\/ferentiating both sides of \eqref{1.1} $r_0+1$ times and using \eqref{B5} for $r= r_0$, we see also that the derivative
 $\frac{\p^{r_0+2}}{\p t^{r_0+2}}v$ exists and is continuous. Furthermore, dif\/ferentiating \eqref{1.1} $r_0+1$ times with respect to $t$
 and once or twice with respect to $x$, from the last statement of our proposition (for the case $r=r_0$) we derive that the derivatives
 $\left(\frac{\p^{r_0+2}}{\p t^{r_0+2}}v\right)_{x}$ and $\left(\frac{\p^{r_0+2}}{\p t^{r_0+2}}v\right)_{xx}$ exist and are continuous
 in $\cD(\ve_{4(r_0+2)})$. Now, we see that the conditions of Proposition \ref{PnB2}
are fulf\/illed for $f=\frac{\p^{r_0+1}}{\p t^{r_0+1}}v$, and so the f\/irst equality in~\eqref{B5} holds for $k\leq r_0+1$.
Using this f\/irst equality in \eqref{B5}, we derive that the conditions of Proposition~\ref{PnB2}
are fulf\/illed for $f=\left(\frac{\p^{r_0+1}}{\p t^{r_0+1}}v\right)_x$
and therefore the second equality in~\eqref{B5} holds for $k\leq r_0+1$.
Dif\/ferentiating again both sides of~\eqref{1.1}, we show that the last statement in Proposition~\ref{prequa} holds for $r =r_0+1$.
\end{proof}

The class $C\big(\{\wt M_k\}\big)$
consists of all inf\/initely dif\/ferentiable on $[0, a)$ scalar functions $f$ such that for some
$c(f) \geq 0$ and for f\/ixed constants $\wt M_k >0 $ $(k \geq 0)$ we have
\begin{gather*} 
\left| \frac{d^k f}{dx^k}(x)\right|\leq c(f)^{k+1}\wt M_k \qquad \text{for all} \quad x\in [0, a).
\end{gather*}
Here, we use the notation $\wt M_k$ (as well $\wt M$ below) because the upper estimates~$M$ (without tilde) were already used
in Section~\ref{Prel}.
Recall that $C\big(\{\wt M_k\}\big)$ is called quasi-analytic if
for the func\-tions~$f$ from this class and for any $0\leq x <a$ the equalities $\frac{d^k f}{dx^k}(x)=0$ $(k\geq 0)$ yield $f\equiv 0$.
According to the famous Denjoy--Carleman theorem, the equality
\begin{gather*} 
\sum_{n=1}^{\infty}\frac{1}{L_n}=\infty, \qquad L_n:=\inf_{k\geq n}\wt M_k^{1/k}
\end{gather*}
implies that the class $C\big(\{\wt M_k\}\big)$ is quasi-analytic.
\begin{Corollary}\label{CyRec} If $v(x,t)$ satisfies conditions of Proposition~{\rm \ref{prequa}} and the entries of $v(0,t)$
or $v_x(0,t)$ are quasi-analytic, then the matrix functions $v(0,t)$ or $v_x(0,t)$, respectively, are uniquely defined
by the initial condition~\eqref{IC1}.
\end{Corollary}

Let us consider the case, where both matrix functions $v(0,t)$ and $v_x(0,t)$ are quasi-analytic.
More precisely, we assume that the entries $v_{ij}(0,t)$ of $v(0,t)$ belong to some quasi-analytic classes $C(\{\wt M_k(i,j)\}) $,
the entries $ \big(v_{ij}\big)_x(0,t) $ of $ v_x(0,t) $ belong to some quasi-analytic classes $C(\{\wt M_k^{+}(i,j)\})$, and, in this case,
we say that $v(0,t) \in C([0,a); \wt M)$ and $v_x(0,t) \in C([0,a); \wt M^+)$, where
\begin{gather*}
\wt M=\{\wt M_k(i,j)\} \qquad {\mathrm{and}} \qquad \wt M^+=\{\wt M^+_k(i,j)\}.
\end{gather*}

Now, using Proposition \ref{PnW1}, Theorem \ref{evol}, Remark~\ref{RkInvPr} and Corollary~\ref{CyRec} we obtain the main theorem
in this section.
\begin{Theorem} Assume that $v\in C_{\ve}(\cD)$ satisfies the dNLS equation \eqref{1.1} on $\cD$, that~\eqref{2.7} holds
 and that boundary values $v(0,t)$ and $v_x(0,t)$ belong to quasi-analytic classes $ C([0,a); \wt M)$ and $C([0,a); \wt M^+)$,
 respectively. Then, $v$ is uniquely defined by the initial condition~\eqref{IC1}.
\end{Theorem}

\begin{Remark}\label{Rec} We see that the scheme to recover $v$ (in the semi-strip $\cD$) from the initial condition
follows from Proposition \ref{PnW1}, Theorem \ref{evol} and the proof of Proposition \ref{prequa}.
The only step that we did not describe in detail is the recovery of the functions $v(0,t)$ and $v_x(0,t)$ from their Taylor coef\/f\/icients
at $t=0$.
Although Taylor coef\/f\/icients uniquely determine quasi-analytic functions $v(0,t)$ and $v_x(0,t)$,
the recovery of these functions presents an interesting problem, which is not solved completely so far.
See \cite{Bang, Khr} and \cite[Section~III.8]{Beur} for some important results.
\end{Remark}

Another important case, where the boundary conditions of the nonlinear Schr\"odinger equation determine a quasi-analytic initial condition, is
discussed in \cite[Section~3]{Sa14b}. We note that
 our solutions are not (in general) quasi-analytic.
\begin{Remark} An interesting class of such solutions of a scalar dNLS that the Weyl functions
$\wt \vp(t,z)$ may be presented as the series $\wt \vp(t,z)=\sum\limits_{k=0}^{\infty}\a_k(t)/z^k$ $($for suf\/f\/iciently large values of
$z)$ was treated in~\cite{SaL91}. $($We note that the Weyl functions $\wt \vp(z)$ from \cite{SaL91}
are Herglotz functions and can be easily mapped into the Weyl functions $\vp(z)$ considered here via a linear-fractional transformation
with constant coef\/f\/icients.$)$ According to~\cite[Theorem~1]{SaL91}, if $\wt \vp(0,z)=\sum\limits_{k=0}^{\infty}\a_k(0)/z^k$, there is one and
only one solution of dNLS from this class in some semi-strip.
\end{Remark}

Important results on the asymptotics of Weyl functions
are given, for instance, in \cite{ClGe, Harr}. However, it would be fruitful to know, under which conditions the asymptotic series (for Weyl functions)
from \cite{ClGe, Harr} converge or at least uniquely def\/ine the corresponding Weyl function.

\section{Nonlinear optics equation on a semi-strip} \label{NW}

In this section we consider the nonlinear optics equation \eqref{NW1}, where $D$ has the form \eqref{NW2}.
We consider equation \eqref{NW1} on the semi-strip $\cD$, which is given by~\eqref{1.4}. First, using Weyl theoretic
results from Section~\ref{subs2.2}, we express Weyl function $\vp(t,z)$ of the auxiliary system~\eqref{NW5},
where $\zeta(x)=\zeta(x,t)=[D,\vr(x,t)]$, in terms of $\vp(0,z)$ and
the boundary condition $\vr(0,t)=\wh \rho(t)$. In other words, we express in these terms the evolution of the Weyl function.
For that purpose, following formulas~\eqref{2.4} and~\eqref{2.5} in Proposition~\ref{TmM},
we introduce matrix function $R(t,z)$ by the relations
\begin{gather} \label{NW20}
R_t(t,z)=(\I z \wh D-[\wh D,\wh \rho(t)])R(t,z), \qquad R(0,z)=I_m.
\end{gather}
Recall that $\wh D$ is a diagonal matrix and that $\wh D>0$.
\begin{Theorem}
\label{TmNWevol}
Let $\vr(x,t)$ satisfy the nonlinear optics equation \eqref{NW1} and the boundary condition $\vr(0,t)=\wh \rho(t)$.
Assume that $\vr$ is uniformly bounded and continuously differentiable on~$\cD$.
Then, the matrix functions
\begin{gather} \label{NW21}
\breve \psi_k(t,z):=
 \begin{bmatrix} I_k & 0\end{bmatrix}
R(t,z) \varphi (0,z)
 \left[
  \begin{matrix}
    0  \\
    I_{m-k}
  \end{matrix}
 \right]
 \biggl(
 \begin{bmatrix} 0 & I_{m-k}\end{bmatrix}
  R(t,z) \varphi (0,z)
  \left[
  \begin{matrix}
    0  \\
    I_{m-k}
  \end{matrix}
  \right]
 \biggr)^{-1}
\end{gather}
are well-defined for $1\leq k <m$, and the evolution of the Weyl function is given by the formula
\begin{gather} \label{NW22}
  \left\{
   \varphi_{i, k+1}(t,z)
  \right\}_{i=1}^k
 = \breve \psi_k (t,z)\begin{bmatrix} 1\\ 0 \\ \ldots \\0
 \end{bmatrix}, \qquad \Im(z)<-M
\end{gather}
and by the normalization conditions~\eqref{NWd7}.
\end{Theorem}

\begin{proof} We set
\begin{gather} \label{NW23}
M_0=\sup\| \zeta(x,t)\|, \qquad (x,t)\in \cD, \qquad M>\max_{1\leq k<m}\big(2M_0/(d_k-d_{k+1})\big).
\end{gather}
Recall that $G$ and $F$ for the case of the nonlinear optics equation are given by~\eqref{NW4}. Since $\vr$ is continuously dif\/ferentiable, the conditions of Proposition~\ref{TmM}
are fulf\/illed. Taking into account that the fundamental solution of~\eqref{NW5} is denoted by~$w$
(instead of~$u$ in Proposition~\ref{TmM}), we rewrite~\eqref{2.6} in the form
$w(x,t,z)=R(x,t,z)w(x,0,z)R(t,z)^{-1}$ or, equivalently,
\begin{gather} \label{NW24}
w(x,t,z)^{-1}=R(t,z)w(x,0,z)^{-1}R(x,t,z)^{-1}.
\end{gather}
Using \eqref{NW24}, we express $w(x,t,z)^{-1}$ via $w(x,\wt t,z)^{-1}$, $0 \leq t,\wt t <a$:
\begin{gather}
w(x,t,z)^{-1} =R(t,z)R(\wt t,z)^{-1}R(\wt t,z)w(x,0,z)^{-1}R(x,\wt t,z)^{-1}\big(R(x,t,z)R(x,\wt t,z)^{-1}\big)^{-1}\nonumber\\
 \label{NW25}
 \hphantom{w(x,t,z)^{-1}}{}
 =R(t,z)R(\wt t,z)^{-1}w(x,\wt t,z)^{-1}\big(R(x,t,z)R(x,\wt t,z)^{-1}\big)^{-1}.
\end{gather}
Recall that $R_t=FR$, where $F$ is given in \eqref{NW4} and $\wh D>0$. Hence, putting $\cR(x,t, \wt t, z):=\big(R(x,t,z)R(x,\wt t,z)^{-1}\big)^{-1}$ we derive
\begin{gather*}
\frac{\prt}{\prt t} \cR(x,t, \wt t, z)=-\cR(x,t, \wt t, z)F(x,t,z), \\
\frac{\prt}{\prt t} \big(\cR(x,t, \wt t, z)\cR(x,t, \wt t, z)^*)<0, \qquad \Im(z)<0 , \qquad
 \cR(x,\wt t, \wt t, z)=I_m.
\end{gather*}
From the relations above it is immediate that
\begin{gather} \label{NW26}
\|\cR(x, t, \wt t, z)\|<1 \qquad {\mathrm{for}}\quad t>\wt t, \qquad \Im(z)<0,
\end{gather}
which allows us to estimate the dif\/ference
\begin{gather} \label{NW27}
I_m-\cR(x, t, \wt t, z)=\int_{\wt t}^t \cR(x, s, \wt t, z)F(x,s,z)ds.
\end{gather}
According to \eqref{NW23}, \eqref{NW26} and \eqref{NW27}, for each $\delta >0$ and $c>M$ there is $\ve=\ve(M_0)>0$ such that
\begin{gather} \label{NW28}
\| I_m-\cR(x, t, \wt t, z)\| \leq \delta
\\ \nonumber \text{for all} \quad x\in [0, \infty), \qquad \! 0\leq t-\wt t \leq \ve,\qquad\!
 z \in \{z \colon |z|<c\}\cap \{z \colon \Im(z)<-M\}, \qquad\! c>M .
\end{gather}
Modifying~\eqref{NW13} (so that the functions $\psi_k$ and $w$ depend there also on an additional variable~$t$),
in view of~\eqref{NW25},
we derive
\begin{gather}\nonumber
  \psi_k (x,t,z)
= \begin{bmatrix}I_k &0 \end{bmatrix}
R(t,z)R(\wt t,z)^{-1}w(x,\wt t,z)^{-1}\cR(x,t, \wt t,z)
 \cQ_k (z)
\\
\hphantom{\psi_k (x,t,z)=}{} \times \left(
  \begin{bmatrix}0 &I_{m-k} \end{bmatrix}
 R(t,z)R(\wt t,z)^{-1}w(x,\wt t,z)^{-1}\cR(x,t, \wt t,z)
  \cQ_k(z)
 \right)^{-1}.
 \label{NW30}
\end{gather}
Moreover, putting
\begin{gather} \label{NW31}
\wt \cQ_k(z):=\cR(x,t, \wt t,z)\cQ_k(z), \qquad \cQ_k(z):=\begin{bmatrix} 0 & I_{m-k}\end{bmatrix}
\end{gather}
we see that for suf\/f\/iciently small $\delta$ the matrix function $\wt \cQ_k(z)$ satisf\/ies \eqref{NW12} in the domain (for~$z$)
given in~\eqref{NW28}. Substituting $\wt \cQ_k$ (instead of $ \cQ_k$) into~\eqref{NW13}, we obtain
\begin{gather} \nonumber
\begin{bmatrix}\psi_k(x, \wt t,z) \\ I_{m-k} \end{bmatrix}=w(x,\wt t,z)^{-1}\wt \cQ_k(z)\left(
  \begin{bmatrix}0 &I_{m-k} \end{bmatrix}
 w(x,\wt t,z)^{-1}
  \wt \cQ_k(z)
 \right)^{-1}, \qquad {\mathrm{i.e.,}}
\\ \label{NW32}
w(x,\wt t,z)^{-1}\wt \cQ_k(z)=\begin{bmatrix}\psi_k(x, \wt t,z) \\ I_{m-k} \end{bmatrix}\left(
  \begin{bmatrix}0 &I_{m-k} \end{bmatrix}
 w(x,\wt t,z)^{-1}
  \wt \cQ_k(z)
 \right).
\end{gather}
Using the f\/irst equality in \eqref{NW31}, we can substitute \eqref{NW32} into \eqref{NW30}.
Thus, we derive
\begin{gather*}
\begin{bmatrix} \psi_k(x, t,z) \\ I_{m-k} \end{bmatrix}
=
R(t,z)R(\wt t,z)^{-1}\begin{bmatrix}\psi_k(x, \wt t,z) \\ I_{m-k} \end{bmatrix} \left(
  \begin{bmatrix}0 &I_{m-k} \end{bmatrix}
 R(t,z)R(\wt t,z)^{-1}\begin{bmatrix}\psi_k(x, \wt t,z) \\ I_{m-k} \end{bmatrix} \right)^{-1}
\end{gather*}
and in view of \eqref{NW15}, passing to the limit $x \to \infty$, we have the following formula{\samepage
\begin{gather}\nonumber
\begin{bmatrix}\breve \psi_k( t,z) \\ I_{m-k} \end{bmatrix}
=
R(t,z)R(\wt t,z)^{-1}\begin{bmatrix}\breve \psi_k( \wt t,z) \\ I_{m-k} \end{bmatrix}
\\ \label{NW33}
\hphantom{\begin{bmatrix}\breve \psi_k( t,z) \\ I_{m-k} \end{bmatrix}
= }{}\times
 \left(
  \begin{bmatrix}0 &I_{m-k} \end{bmatrix}
 R(t,z)R(\wt t,z)^{-1}\begin{bmatrix}\breve \psi_k(\wt t,z) \\ I_{m-k} \end{bmatrix} \right)^{-1}.
\end{gather}}

\noindent
Here we used the fact that, according to \eqref{NW14} and \eqref{NW28},
\begin{gather*}
\det \left(
  \begin{bmatrix}0 &I_{m-k} \end{bmatrix}
 R(t,z)R(\wt t,z)^{-1}\begin{bmatrix}\breve \psi_k(\wt t,z) \\ I_{m-k} \end{bmatrix} \right)\not=0
\end{gather*}
for suf\/f\/iciently small $\delta$. Setting $\wt t=k\ve \,\, (k=0,1, \ldots)$, recalling that $R(0,z)=I_m$ and applying each time~\eqref{NW33}, we easily prove (by induction) the equality
\begin{gather} \label{NW34}
\begin{bmatrix}\breve \psi_k( t,z) \\ I_{m-k} \end{bmatrix}
=
R(t,z)\begin{bmatrix}\breve \psi_k(0,z) \\ I_{m-k} \end{bmatrix}
 \left(
  \begin{bmatrix}0 &I_{m-k} \end{bmatrix}
 R(t,z)\begin{bmatrix}\breve \psi_k(0,z) \\ I_{m-k} \end{bmatrix} \right)^{-1}
\end{gather}
for $t$ on all intervals $[0, (k+1) \ve]\cap [0,a)$, that is, for $t$ on $[0,a)$. Although
\eqref{NW34} is proved for $ z \in \{z \colon |z|<c\}\cap \{z \colon \Im(z)<-M\}$, the analyticity
of both sides of \eqref{NW34} implies that the equality holds in the semi-plane $\Im(z)<-M$.
Finally, we note that~\eqref{NW17} at $t=0$ yields
\begin{gather} \label{NW35}
\begin{bmatrix}\breve \psi_k( 0,z) \\ I_{m-k} \end{bmatrix}
=
\vp(0,z)\begin{bmatrix}0 \\ I_{m-k} \end{bmatrix}
 \left(
  \begin{bmatrix}0 &I_{m-k} \end{bmatrix}
 \vp(0,z)\begin{bmatrix} 0 \\ I_{m-k} \end{bmatrix} \right)^{-1}.
\end{gather}
Substituting \eqref{NW35} into~\eqref{NW34}, we obtain~\eqref{NW21}. The procedure to recover
$\vp(t,z)$ from $\{\breve \psi_k(t,z)\}$ (for f\/ixed values of $t$) is described in Section~\ref{subs2.2}
(note that~\eqref{NW22} coincides with~\eqref{NW16}).
\end{proof}

Now, let us prove a uniqueness result for the case of $\wh D$ of the form \eqref{NW3}.
Let initial condition be given by the equality
\begin{gather*} 
\vr(x,0)= \rho(x), \qquad \sup_{x\in [0,\infty)} \|\rho(x)\|<\infty.
\end{gather*}
Denote the Weyl function of system
\begin{gather} \label{NW37}
y_x(x,z)=(\I z D-\zeta(x))y(x,z), \qquad x \geq 0 , \qquad \zeta=[D, \rho]
\end{gather}
by $\vp_0(z)$. (According to Section~\ref{subs2.2}, this Weyl function exists and is unique.)
\begin{Theorem}\label{TmNWnm}
For the case where the entries of the matrix $\wh D$ in \eqref{NW1} are ordered as in \eqref{NW3},
there is no more than one uniformly bounded and continuously
differentiable on $\cD$ solution $\vr=\vr^*$ $($of the nonlinear optics equation~\eqref{NW1}$)$,
having
the initial values~$\vr(x,0)$ such
that
$\vp_0$ is bounded and~\eqref{NW8} holds. That is, there is no more than one
solution of the corresponding initial value problem.
\end{Theorem}

\begin{proof}
Let $\vr(x,t)$ satisfy conditions of the theorem. We f\/ix $M$ such~\eqref{NW23} is valid, $\vp_0(z)$ is bounded for $\Im(z)\leq -M$ and \eqref{NW8} holds for $\vp_0(\xi+\I \eta)$ when $\eta \leq -M$. We
set also
\begin{gather} \label{NW38}
\wh M=\sup_{0\leq t<a}\| \wh \zeta(0,t)\|, \qquad \wh \zeta(0,t)=[\wh D, \wh \rho(t)],
\end{gather}
where $\wh \rho(t)=\vr(0,t)$. First, we show that the inequality
\begin{gather} \label{NW38!}
\sup_{t\in [0,a), \, \Im(z)<-M}\|R(t,z)\vp_0(z)\exp\{-\I z t \wh D\}\|<\infty, \qquad a<\infty
\end{gather}
is valid. Indeed, according to \eqref{NW14} and \eqref{NW34} we have
\begin{gather} \label{NW39}
\begin{bmatrix}\breve \psi_k( 0,z)^* & I_{m-k} \end{bmatrix} R(t,z)^* J_k
R(t,z)\begin{bmatrix}\breve \psi_k(0,z) \\ I_{m-k} \end{bmatrix}
\leq 0.
\end{gather}
Clearly, \eqref{NW39} yields the inequality
\begin{gather} \nonumber
 - 4 \wh M \begin{bmatrix}\breve \psi_k( 0,z)^* & I_{m-k} \end{bmatrix} R(t,z)^* (I_m-J_k)
R(t,z)\begin{bmatrix}\breve \psi_k(0,z) \\ I_{m-k} \end{bmatrix}
\\ \label{NW40} \qquad{}
\leq - 4 \wh M \begin{bmatrix}\breve \psi_k( 0,z)^* & I_{m-k} \end{bmatrix} R(t,z)^*
R(t,z)\begin{bmatrix}\breve \psi_k(0,z) \\ I_{m-k} \end{bmatrix} .
\end{gather}
Taking into account that $R_t=FR$ and relations \eqref{NW3}, \eqref{NW4}, \eqref{NW38} and \eqref{NW40} hold, we derive
\begin{gather}
\frac{d}{d t}\bigg(\exp\{\I(\ov z -z)\wh d_{k+1} t - 4 \wh M t\}
\begin{bmatrix} \breve \psi_k( 0,z)^* & I_{m-k} \end{bmatrix}\nonumber\\
\qquad{}\times R(t,z)^* (I_m-J_k )
R(t,z)\begin{bmatrix}\breve \psi_k(0,z) \\ I_{m-k} \end{bmatrix} \bigg)
\leq 0\label{NW41}
\end{gather}
for $\Im(z)<0$. Formulas \eqref{NW39} and \eqref{NW41} imply that
\begin{gather} \label{NW42}
\exp\big\{\I(\ov z -z)\wh d_{k+1} t - 4 \wh M t\big\}
\begin{bmatrix} \breve \psi_k( 0,z)^* & I_{m-k} \end{bmatrix} R(t,z)^*
R(t,z)\begin{bmatrix}\breve \psi_k(0,z) \\ I_{m-k} \end{bmatrix}
\\ \qquad{}
\leq
\begin{bmatrix} \breve \psi_k( 0,z)^* & I_{m-k} \end{bmatrix} (I_m-J_k )
\begin{bmatrix}\breve \psi_k(0,z) \\ I_{m-k} \end{bmatrix}
=2 I_{m-k}, \qquad 1\leq k<m, \quad \Im(z)<-M.\nonumber
\end{gather}
Recall that $\vp_0$ is given by \eqref{NW16}. Hence, \eqref{NW38!} follows from~\eqref{NW42}.

Consider system \eqref{NW20}. Since $\vp_0$ is bounded and satisf\/ies \eqref{NW8} and \eqref{NW38!}, according to Corollary~\ref{CyUniq},
the matrix function
$\wh \zeta(0,t)=[\wh D,\wh \rho(t)]$ (and so $R$) is uniquely def\/ined by $\vp_0$. Thus, from Theorem~\ref{TmNWevol}, we see that
$\vp(t,z)$ is uniquely def\/ined by~$\vp_0$.

In order to prove our theorem, it remains to show that $\vp(t,z)$ satisf\/ies conditions of Theo\-rem~\ref{NWTmUniq}
for each $t$. Indeed, in view of~\eqref{NW20}, we can rewrite for~$R$ the representation~\eqref{NW18}:
\begin{gather}
  R(t,z) = \exp \{ \I z t \wh D\}
    + \int_{\wh d_m t}^{\wh d_1 t}
           \exp \{\I zs\}
     \wh N(t,s)
    \, ds, \qquad \sup_{t<a} \| \wh N(t,s) \| < \infty.
 \label{NW43}
\end{gather}
By virtue of \eqref{NW38!}, the matrix function $\, \,R(t,z)\, \vp_0(z) \, \E^{-\I z t \wh D} -I_m \,\,$ is bounded in the domain
$\Im(z)\leq -M$.
Since $R(t,z)$ satisf\/ies~\eqref{NW43} and $M$ is chosen so that $\vp_0$ satisf\/ies \eqref{NW8} for $z=\xi -\I M$,
we see that $R(t,\xi-\I M)\vp_0(\xi-\I M)\E^{-\I (\xi-\I M)t \wh D} -I_m \in L^2_{m\times m}(-\infty, \infty)$
for each $0\leq t <a$,
where $L^2_{m\times m}(0,\infty)$ is the class of $m\times m$ matrix functions, the entries of which belong to $L^2(0,\infty)$.
 Hence, the well-known Theorems~V and~VIII (Sections~4 and~5 in~\cite{WP}, respectively) on the Fourier transform in complex domains
yield the
representation (see \cite[formula~(E11)]{SaSaR}):
\begin{gather}
   \label{NW44}
R(t,z)\vp_0(z)\E^{-\I z t \wh D}=I_m+\int_0^{\infty}\E^{-\I zx}\cF(x)dx, q\quad \E^{-xM}\cF(x)\in L^2_{m\times m}(0,\infty) .
\end{gather}
It is immediate also that formula \eqref{NW21} can be modif\/ied slightly:
\begin{gather}
\breve \psi_k(t,z)=
 \begin{bmatrix} I_k & 0\end{bmatrix}
R(t,z) \varphi_0 (z)\E^{-\I z t \wh D}
 \left[
  \begin{matrix}
    0  \\
    I_{m-k}
  \end{matrix}
 \right]\nonumber\\
 \hphantom{\breve \psi_k(t,z)=}{}\times
 \left(
 \begin{bmatrix} 0 & I_{m-k}\end{bmatrix}
  R(t,z) \varphi_0 (z)\E^{-\I z t \wh D}
  \left[
  \begin{matrix}
    0  \\
    I_{m-k}
  \end{matrix}
  \right]
 \right)^{-1} .\label{NW45}
\end{gather}
According to \eqref{NW44}, the normalized GW-function $\vp(t,z)$ constructed via equalities~\eqref{NW22}
and \eqref{NW45} satisf\/ies conditions of Theorem~\ref{NWTmUniq}. In other words, there is no more then
one solution of ISpP for $\vp(t,z)$, that is, $\vr(x,t)$ is unique.
\end{proof}
\begin{Remark}\label{RkDop}
In the case of system \eqref{NW20}, by virtue of
 \eqref{NW38} and \eqref{NW38!}, the requirements of Corollary \ref{CyUniq} are fulf\/illed for~$\vp_0$,
 and so $\vp_0$ uniquely determines the boundary condition~$\wh \rho$.
 Moreover, if $\vp_0$ satisf\/ies~\eqref{NW9} and \eqref{NW10} there is a rigorous procedure to recover
 $\wh \rho $ from $\vp_0$ $($see Remark~\ref{NWInvPr}$)$.
\end{Remark}

Although Theorem \ref{TmNWnm} was announced in \cite{SaARMS}, its proof is published for the f\/irst time.
It is essential to know, for which initial conditions $\rho(x)$, the restrictions on $\vp_0$ (from Theorem~\ref{TmNWnm})
are fulf\/illed. First, let us formulate a particular case of Theorem~6.1 from~\cite{BC0}.
\begin{Proposition} \label{NWPnBC} Suppose that the $m \times m$ matrix function $\rho(x)$ is absolutely continuous on $\BR$
and $\rho(x), \rho^{\prime}(x)\in
L^1_{m\times m}(-\infty, \ \infty)$. Then, for some $M>0$, there is an
analytic with respect to~$z$ fundamental $($unnormalized$)$ solution~$\cM(x,z)$
of the equation
\begin{gather}
 \cM_x=\I z [D, \cM]-\zeta \cM, \qquad x \in \BR, \qquad \zeta= [D, \rho] ,
 \label{NW46}
\end{gather}
such that uniformly with respect to $x$ we have
\begin{gather}
 \label{NW47}
\cM(x,z)=I_m+\frac{1}{z}\cM_1(x)+o\big(|z|^{-1}\big), \qquad |z| \to \infty, \qquad \Im z<-M,
\end{gather}
where $\cM_1(x)$ is absolutely continuous.
\end{Proposition}

We note that the fact that $\cM_1$ is absolutely continuous is immediate from the proof of
\cite[Theorem~6.1]{BC0} (more precisely, from formulas (6.6) and (6.8)). Now, since we can always
extend $\rho^{\prime}$ on $\BR$, we consider $\rho(x)$ on $[0,\infty)$ only and assume that $\rho$ is absolutely continuous
and $\rho, \, \rho^{\prime}\in
L^1_{m\times m}(\BR_+)$. Without loss of generality, we assume that $M$ is chosen so that $\cM(0,z)$ is invertible for
$\Im(z)<-M$. Then, according to \eqref{NW46}, the matrix function
\begin{align}&
w(x,z)=\cM(x,z)\E^{\I z x D}\cM(0,z)^{-1}
 \label{NW48}
\end{align}
is the normalized (by $w(0,z)=I_m$) fundamental solution of the equation \eqref{NW37} (and we don't require so far
that $\rho=\rho^*$).
Moreover, \eqref{NW47} implies that~\eqref{NW7} holds for $\vp(z)=\cM(0,z)$. That is, assuming $\rho=\rho^*$
and taking into account Def\/inition~\ref{NWDn2}, we see that
$\cM(0,z)$ is a GW-function of~\eqref{NW37}.

Recall that in Theorem~\ref{TmNWnm} we speak about the Weyl function $\vp_0$ or, equivalently for a~boun\-ded function $\rho=\rho^*$,
about the normalized GW-function. Thus, we should normali\-ze~$\cM(0,z)$. For that purpose we construct a lower triangular
matrix function $\wh \cM(z)$ via the right lower $k\times k$ blocks $\cP_k(z)$ of $\cM(0,z)$. Namely, we construct
$\wh \cM(z)$ columnwise via the equalities
\begin{gather}
\wh \cM(z)\{\delta_{i,m-k+1}\}_{i=1}^m:=\begin{bmatrix}
0 \\ \cP_k(z)^{-1}\{\delta_{i1}\}_{i=1}^k
\end{bmatrix}, \qquad 1\leq k \leq m ,
 \label{NW49}
\end{gather}
where $\{\delta_{i,m-k+1}\}_{i=1}^m$ and $\{\delta_{i1}\}_{i=1}^k$ are column vectors.
It follows from~\eqref{NW49} that the normalization conditions~\eqref{NWd7} hold for
\begin{gather}
\vp_0(z)=\cM(0,z)\wh \cM(z).
 \label{NW50}
\end{gather}
Since $\cM(z)$ is lower triangular and $D$ satisf\/ies \eqref{NW2}, we see that $\E^{\I z x D}\wh \cM(z) \E^{-\I z x D}$ is bounded
for $\Im(z)<-M$. Hence, taking into account that $\cM(0,z)$ is a GW-function, we derive
that \eqref{NW7} is also valid for $\vp(z)=\cM(0,z)\wh \cM(z)$ (i.e., $\vp_0$ given by~\eqref{NW50} is the normalized GW-function).
Finally, relations \eqref{NW47}, \eqref{NW49} and \eqref{NW50} show that $\vp_0$ is bounded and that~\eqref{NW8} also holds
for $\vp_0$. Thus, we proved the statement below.
\begin{Proposition} \label{NWPnNorm} Suppose that the initial condition $\rho(x)=\rho(x)^*$ is absolutely continuous on $[0, \infty)$
and $\rho(x), \rho^{\prime}(x)\in
L^1_{m\times m}(\BR_+)$. Then, the Weyl function $\vp_0(z)$ of the system~\eqref{NW37}, where $\zeta=\zeta_0=[D,\rho]$, exists. Moreover, $\vp_0(z)$ is
analytic and bounded $($in some semi-plane $\Im(z)<-M$, $M>0)$, and it satisfies~\eqref{NW8}.
\end{Proposition}

An existence result for a solution of an initial value problem (for the nonlinear optics equation)
is given in \cite[Remark~4.7]{SaA17}.

\subsection*{Acknowledgements}

This research was
supported by the Austrian Science Fund (FWF) under Grant No.~P24301.
The author is grateful to A.~Rainer for a helpful discussion on
quasi-analytic functions.

\pdfbookmark[1]{References}{ref}
\LastPageEnding


\begin{thebibliography}{99}
\footnotesize\itemsep=0pt

\bibitem{AbH}
Ablowitz M.J., Haberman R., Resonantly coupled nonlinear evolution equations,
 \href{http://dx.doi.org/10.1063/1.522460}{\textit{J.~Math. Phys.}} \textbf{16} (1975), 2301--2305.

\bibitem{Abl0}
Ablowitz M.J., Prinari B., Trubatch A.D., Discrete and continuous nonlinear
 {S}chr\"odinger systems, \textit{London Mathematical Society Lecture Note
 Series}, Vol.~302, Cambridge University Press, Cambridge, 2004.

\bibitem{Ash}
Ashton A.C.L., On the rigorous foundations of the {F}okas method for linear
 elliptic partial dif\/ferential equations, \href{http://dx.doi.org/10.1098/rspa.2011.0478}{\textit{Proc.~R. Soc. Lond. Ser.~A
 Math. Phys. Eng. Sci.}} \textbf{468} (2012), 1325--1331.

\bibitem{Bang}
Bang T., On quasi-analytic functions, in C.~{R}.~{D}ixi\`eme {C}ongr\`es
 {M}ath. {S}candinaves 1946, Jul. Gjellerups Forlag, Copenhagen, 1947,
 249--254.

\bibitem{BC0}
Beals R., Coifman R.R., Scattering and inverse scattering for f\/irst order
 systems, \href{http://dx.doi.org/10.1002/cpa.3160370105}{\textit{Comm. Pure Appl. Math.}} \textbf{37} (1984), 39--90.

\bibitem{Ber}
Berezanskii Yu.M., Integration of non-linear dif\/ference equations by means of
 inverse problem technique, \textit{Dokl. Akad. Nauk SSSR} \textbf{281}
 (1985), 16--19.

\bibitem{BerG}
Berezanskii Yu.M., Gekhtman M.I., Inverse problem of spectral analysis and
 nonabelian chains of nonlinear equations, \href{http://dx.doi.org/10.1007/BF01058907}{\textit{Ukrain. Math.~J.}}
 \textbf{42} (1990), 645--658.

\bibitem{Beur}
Beurling A., The collected works of {A}rne {B}eurling. {V}ol.~1. Complex
 analysis, \textit{Contemporary Mathematicians}, Birkh\"auser Boston, Inc., Boston, MA,
 1989.

\bibitem{BiTa}
Bikbaev R.F., Tarasov V.O., Initial-boundary value problem for the nonlinear
 {S}chr\"odinger equation, \href{http://dx.doi.org/10.1088/0305-4470/24/11/017}{\textit{J.~Phys.~A: Math. Gen.}} \textbf{24} (1991),
 2507--2516.

\bibitem{BW83}
Bona J., Winther R., The {K}orteweg--de {V}ries equation, posed in a
 quarter-plane, \href{http://dx.doi.org/10.1137/0514085}{\textit{SIAM~J. Math. Anal.}} \textbf{14} (1983), 1056--1106.

\bibitem{BoFo}
Bona J.L., Fokas A.S., Initial-boundary-value problems for linear and
 integrable nonlinear dispersive partial dif\/ferential equations,
 \href{http://dx.doi.org/10.1088/0951-7715/21/10/T03}{\textit{Nonlinearity}} \textbf{21} (2008), T195--T203.

\bibitem{CB}
Carroll R., Bu Q., Solution of the forced nonlinear {S}chr\"odinger ({NLS})
 equation using {PDE} techniques, \href{http://dx.doi.org/10.1080/00036819108840015}{\textit{Appl. Anal.}} \textbf{41} (1991),
 33--51.

\bibitem{ChuX}
Chu C.K., Xiang L.W., Baransky Y., Solitary waves induced by boundary motion,
 \href{http://dx.doi.org/10.1002/cpa.3160360407}{\textit{Comm. Pure Appl. Math.}} \textbf{36} (1983), 495--504.

\bibitem{ClGe}
Clark S., Gesztesy F., Weyl--{T}itchmarsh {$M$}-function asymptotics, local
 uniqueness results, trace formulas, and {B}org-type theorems for {D}irac
 operators, \href{http://dx.doi.org/10.1090/S0002-9947-02-03025-8}{\textit{Trans. Amer. Math. Soc.}} \textbf{354} (2002), 3475--3534,
 \href{http://arxiv.org/abs/math.SP/0102040}{math.SP/0102040}.

\bibitem{DaKiS}
Damanik D., Killip R., Simon B., Perturbations of orthogonal polynomials with
 periodic recursion coef\/f\/icients, \href{http://dx.doi.org/10.4007/annals.2010.171.1931}{\textit{Ann. of Math.}} \textbf{171} (2010),
 1931--2010, \href{http://arxiv.org/abs/math.SP/0702388}{math.SP/0702388}.

\bibitem{DeMaS}
Degasperis A., Manakov S.V., Santini P.M., Mixed problems for linear and
 soliton partial dif\/ferential equations, \href{http://dx.doi.org/10.1023/A:1021138525261}{\textit{Theoret. and Math. Phys.}}
 \textbf{133} (2002), 1475--1489.

\bibitem{Fok0}
Fokas A.S., Integrable nonlinear evolution equations on the half-line,
 \href{http://dx.doi.org/10.1007/s00220-002-0681-8}{\textit{Comm. Math. Phys.}} \textbf{230} (2002), 1--39.

\bibitem{Fok}
Fokas A.S., A unif\/ied approach to boundary value problems, \href{http://dx.doi.org/10.1137/1.9780898717068}{\textit{CBMS-NSF
 Regional Conference Series in Applied Mathematics}}, Vol.~78, Society for
 Industrial and Applied Mathematics (SIAM), Philadelphia, PA, 2008.

\bibitem{FoP}
Fokas A.S., Pelloni B. (Editors), Unif\/ied transform for boundary value
 problems. Applications and advances, Society for Industrial and Applied
 Mathematics (SIAM), Philadelphia, PA, 2015.

\bibitem{FKRSp1}
Fritzsche B., Kirstein B., Roitberg I.Ya., Sakhnovich A.L., Weyl theory and
 explicit solutions of direct and inverse problems for {D}irac system with a
 rectangular matrix potential, \href{http://dx.doi.org/10.7153/oam-07-10}{\textit{Oper. Matrices}} \textbf{7} (2013),
 183--196, \href{http://arxiv.org/abs/1105.2013}{arXiv:1105.2013}.

\bibitem{FKS}
Fritzsche B., Kirstein B., Sakhnovich A.L., Weyl functions of generalized
 {D}irac systems: integral representation, the inverse problem and discrete
 interpolation, \href{http://dx.doi.org/10.1007/s11854-012-0002-x}{\textit{J.~Anal. Math.}} \textbf{116} (2012), 17--51,
 \href{http://arxiv.org/abs/1007.4304}{arXiv:1007.4304}.

\bibitem{GMZ}
Gesztesy F., Mitrea M., Zinchenko M., On {D}irichlet-to-{N}eumann maps and some
 applications to modif\/ied {F}redholm determinants, in Methods of Spectral
 Analysis in Mathematical Physics, \href{http://dx.doi.org/10.1007/978-3-7643-8755-6_9}{\textit{Oper. Theory Adv. Appl.}}, Vol.~186,
 Birkh\"auser Verlag, Basel, 2009, 191--215, \href{http://arxiv.org/abs/1002.0390}{arXiv:1002.0390}.

\bibitem{GeZi}
Gesztesy F., Weikard R., Zinchenko M., Initial value problems and
 {W}eyl--{T}itchmarsh theory for {S}chr\"odinger operators with
 operator-valued potentials, \href{http://dx.doi.org/10.7153/oam-07-15}{\textit{Oper. Matrices}} \textbf{7} (2013),
 241--283, \href{http://arxiv.org/abs/1109.1613}{arXiv:1109.1613}.

\bibitem{GKS6}
Gohberg I., Kaashoek M.A., Sakhnovich A.L., Scattering problems for a canonical
 system with a pseudo-exponential potential, \textit{Asymptot. Anal.}
 \textbf{29} (2002), 1--38.

\bibitem{Hab}
Habibullin I.T., Backlund transformation and integrable boundary-initial value
 problems, in Nonlinear World, {V}ol.~1 ({K}iev, 1989), World Sci. Publ.,
 River Edge, NJ, 1990, 130--138.

\bibitem{Harr}
Harris B.J., The asymptotic form of the {T}itchmarsh--{W}eyl {$m$}-function
 associated with a {D}irac system, \href{http://dx.doi.org/10.1112/jlms/s2-31.2.321}{\textit{J.~London Math. Soc.}} \textbf{31}
 (1985), 321--330.

\bibitem{Holm}
Holmer J., The initial-boundary-value problem for the 1{D} nonlinear
 {S}chr\"odinger equation on the half-line, \textit{Differential Integral
 Equations} \textbf{18} (2005), 647--668, \href{http://arxiv.org/abs/math.AP/0602152}{math.AP/0602152}.

\bibitem{Kv}
Kac M., van Moerbeke P., A complete solution of the periodic {T}oda problem,
 \textit{Proc. Nat. Acad. Sci. USA} \textbf{72} (1975), 2879--2880.

\bibitem{Kai}
Kaikina E.I., Inhomogeneous {N}eumann initial-boundary value problem for the
 nonlinear {S}chr\"odinger equation, \href{http://dx.doi.org/10.1016/j.jde.2013.07.036}{\textit{J.~Differential Equations}}
 \textbf{255} (2013), 3338--3356.

\bibitem{KamS}
Kamvissis S., Shepelsky D., Zielinski L., Robin boundary condition and shock
 problem for the focusing nonlinear {S}chr\"odinger equation,
 \href{http://dx.doi.org/10.1080/14029251.2015.1079428}{\textit{J.~Nonlinear Math. Phys.}} \textbf{22} (2015), 448--473,
 \href{http://arxiv.org/abs/1412.7636}{arXiv:1412.7636}.

\bibitem{Ka}
Kaup D.J., The forced {T}oda lattice: an example of an almost integrable
 system, \href{http://dx.doi.org/10.1063/1.526136}{\textit{J.~Math. Phys.}} \textbf{25} (1984), 277--281.

\bibitem{KaNe}
Kaup D.J., Newell A.C., The {G}oursat and {C}auchy problems for the
 sine-{G}ordon equation, \href{http://dx.doi.org/10.1137/0134004}{\textit{SIAM~J. Appl. Math.}} \textbf{34} (1978),
 37--54.

\bibitem{KaSt}
Kaup D.J., Steudel H., Recent results on second harmonic generation, in Recent
 Developments in Integrable Systems and {R}iemann--{H}ilbert Problems
 ({B}irmingham, {AL}, 2000), \href{http://dx.doi.org/10.1090/conm/326/05755}{\textit{Contemp. Math.}}, Vol.~326, Amer. Math.
 Soc., Providence, RI, 2003, 33--48.

\bibitem{Khr}
Khryptun V.G., Expansion of functions of quasi-analytic classes in series in
 polynomials, \href{http://dx.doi.org/10.1007/BF01060545}{\textit{Ukrain. Math.~J.}} \textbf{41} (1989), 569--574.

\bibitem{KoSaTe}
Kostenko A., Sakhnovich A., Teschl G., Weyl--{T}itchmarsh theory for
 {S}chr\"odinger operators with strongly singular potentials, \href{http://dx.doi.org/10.1093/imrn/rnr065}{\textit{Int.
 Math. Res. Not.}} \textbf{2012} (2012), 1699--1747.

\bibitem{Krei}
Kreiss H.-O., Initial boundary value problems for hyperbolic systems,
 \href{http://dx.doi.org/10.1002/cpa.3160230304}{\textit{Comm. Pure Appl. Math.}} \textbf{23} (1970), 277--298.

\bibitem{Kri}
Krichever I.M., An analogue of the d'{A}lembert formula for the equations of a
 principal chiral f\/ield and the sine-{G}ordon equation, \textit{Dokl. Akad.
 Nauk SSSR} \textbf{253} (1980), 288--292.

\bibitem{WP}
Paley R.E.A.C., Wiener N., Fourier transforms in the complex domain,
 \textit{American Mathematical Society Colloquium Publications}, Vol.~19,
 Amer. Math. Soc., Providence, RI, 1987.

\bibitem{Sab1}
Sabatier P.C., Elbow scattering and inverse scattering applications to {LK}d{V}
 and {K}d{V}, \href{http://dx.doi.org/10.1063/1.533138}{\textit{J.~Math. Phys.}} \textbf{41} (2000), 414--436.

\bibitem{Sab3}
Sabatier P.C., Generalized inverse scattering transform applied to linear
 partial dif\/ferential equations, \href{http://dx.doi.org/10.1088/0266-5611/22/1/012}{\textit{Inverse Problems}} \textbf{22} (2006),
 209--228.

\bibitem{SaARMS}
Sakhnovich A.L., The {$N$}-wave problem on the half-line, \href{http://dx.doi.org/10.1070/RM1991v046n04ABEH002834}{\textit{Russ. Math.
 Surv.}} \textbf{46} (1991), no.~4, 198--200.

\bibitem{SaAg}
Sakhnovich A.L., The {G}oursat problem for the sine-{G}ordon equation, and an
 inverse spectral problem, \textit{Russ. Math. Iz. VUZ} (1992), no.~11,
 42--52.

\bibitem{SaA7}
Sakhnovich A.L., Second harmonic generation: {G}oursat problem on the
 semi-strip, {W}eyl functions and explicit solutions, \href{http://dx.doi.org/10.1088/0266-5611/21/2/016}{\textit{Inverse
 Problems}} \textbf{21} (2005), 703--716, \href{http://arxiv.org/abs/nlin.SI/0402055}{nlin.SI/0402055}.

\bibitem{SaA17}
Sakhnovich A.L., Weyl functions, the inverse problem and special solutions for
 the system auxiliary to the nonlinear optics equation, \href{http://dx.doi.org/10.1088/0266-5611/24/2/025026}{\textit{Inverse
 Problems}} \textbf{24} (2008), 025026, 23~pages, \href{http://arxiv.org/abs/0708.1112}{arXiv:0708.1112}.

\bibitem{SaAF}
Sakhnovich A.L., On the compatibility condition for linear systems and a
 factorization formula for wave functions, \href{http://dx.doi.org/10.1016/j.jde.2011.11.001}{\textit{J.~Differential Equations}}
 \textbf{252} (2012), 3658--3667.

\bibitem{Sa14}
Sakhnovich A.L., Inverse problem for {D}irac systems with locally
 square-summable potentials and rectangular {W}eyl functions,
 \href{http://dx.doi.org/10.4171/JST/106}{\textit{J.~Spectr. Theory}} \textbf{5} (2015), 547--569, \href{http://arxiv.org/abs/1401.3605}{arXiv:1401.3605}.

\bibitem{Sa14b}
Sakhnovich A.L., Nonlinear {S}chr\"odinger equation in a semi-strip: evolution
 of the {W}eyl--{T}itchmarsh function and recovery of the initial condition
 and rectangular matrix solutions from the boundary conditions,
 \href{http://dx.doi.org/10.1016/j.jmaa.2014.10.012}{\textit{J.~Math. Anal. Appl.}} \textbf{423} (2015), 746--757.

\bibitem{SaSaR}
Sakhnovich A.L., Sakhnovich L.A., Roitberg I.Ya., Inverse problems and nonlinear
 evolution equations. Solutions, Darboux matrices and Weyl--Titchmarsh
 functions, \href{http://dx.doi.org/10.1515/9783110258615}{\textit{De Gruyter Studies in Mathematics}}, Vol.~47, De Gruyter,
 Berlin, 2013.

\bibitem{Sa88}
Sakhnovich L.A., Evolution of spectral data, and nonlinear equations,
 \href{http://dx.doi.org/10.1007/BF01057215}{\textit{Ukrain. Math.~J.}} \textbf{40} (1988), 459--461.

\bibitem{SaL91}
Sakhnovich L.A., Integrable nonlinear equations on the semi-axis,
 \href{http://dx.doi.org/10.1007/BF01067289}{\textit{Ukrain. Math.~J.}} \textbf{43} (1991), 1470--1476.

\bibitem{SaLev2}
Sakhnovich L.A., The method of operator identities and problems in analysis,
 \textit{St.~Petersburg Math.~J.} \textbf{5} (1994), 1--69.

\bibitem{Sa99}
Sakhnovich L.A., Spectral theory of canonical dif\/ferential systems. {M}ethod of
 operator identities, \href{http://dx.doi.org/10.1007/978-3-0348-8713-7}{\textit{Operator Theory: Advances and Applications}},
 Vol.~107, Birkh\"auser Verlag, Basel, 1999.

\bibitem{See}
Seeley R.T., Classroom notes: {F}ubini implies {L}eibniz implies {$F_{yx} =
 F_{xy}$}, \href{http://dx.doi.org/10.2307/2311366}{\textit{Amer. Math. Monthly}} \textbf{68} (1961), 56--57.

\bibitem{Si0}
Simon B., A new approach to inverse spectral theory. {I}.~{F}undamental
 formalism, \href{http://dx.doi.org/10.2307/121061}{\textit{Ann. of Math.}} \textbf{150} (1999), 1029--1057,
 \href{http://arxiv.org/abs/math.SP/9906118}{math.SP/9906118}.

\bibitem{Skl}
Sklyanin E.K., Boundary conditions for integrable equations, \href{http://dx.doi.org/10.1007/BF01078038}{\textit{Funct.
 Anal. Appl.}} \textbf{21} (1987), 164--166.

\bibitem{Te}
Teschl G., Jacobi operators and completely integrable nonlinear lattices,
 \textit{Mathematical Surveys and Monographs}, Vol.~72, Amer. Math. Soc.,
 Providence, RI, 2000.

\bibitem{Ton}
Ton B.A., Initial boundary value problems for the {K}orteweg--de {V}ries
 equation, \href{http://dx.doi.org/10.1016/0022-0396(77)90046-8}{\textit{J.~Differential Equations}} \textbf{25} (1977), 288--309.

\bibitem{ZaMa}
Zakharov V.E., Manakov S.V., The theory of resonance interaction of wave
 packets in nonlinear media, \textit{Soviet Phys. JETP} \textbf{69} (1975),
 1654--1673.

\bibitem{ZS1}
Zakharov V.E., Shabat A.B., Exact theory of two-dimensional self-focusing and
 one-dimensional self-modulation of waves in nonlinear media, \textit{Soviet
 Phys. JETP} \textbf{61} (1971), 62--69.

\bibitem{ZS2}
Zakharov V.E., Shabat A.B., Integration of nonlinear equations of mathematical
 physics by the method of the inverse scattering problem.~{II}, \href{http://dx.doi.org/10.1007/BF01077483}{\textit{Funct.
 Anal. Appl.}} \textbf{13} (1979), 166--174.

\end{thebibliography}
\end{document}